\documentclass{amsart}
\usepackage{amsmath,amsthm,amssymb,amsfonts,amscd}
\usepackage{mathrsfs}
\usepackage{bbding}
\usepackage{graphicx,latexsym}
\usepackage{hyperref}
\usepackage{geometry}
\usepackage{color}
\usepackage{xcolor}

\hypersetup{
    colorlinks,
    linkcolor={red!50!black},
    citecolor={blue!50!black},
    urlcolor={blue!80!black}
}

\numberwithin{equation}{section}

\theoremstyle{plain}
\newtheorem{theorem}[equation]{Theorem}
\newtheorem*{theorem*}{Theorem}
\newtheorem{lemma}[equation]{Lemma}

\newtheorem{proposition}[equation]{Proposition}

\theoremstyle{definition}

\theoremstyle{remark}
\newtheorem{remark}[equation]{Remark}

\renewcommand{\Re}{\operatorname{Re}}
\renewcommand{\Im}{\operatorname{Im}}

\newcommand{\spec}{\operatorname{spec}}
\newcommand{\sgn}{\operatorname{sgn}}

\renewcommand{\mod}{\operatorname{mod}\,}

\newcommand{\cC}{\mathcal{C}}

\newcommand{\cE}{\mathcal{E}}

\newcommand{\cN}{\mathcal{N}}

\newcommand{\cW}{\mathcal{W}}

\newcommand*{\bbC}{\ensuremath{\mathbb{C}}}

\newcommand*{\bbR}{\ensuremath{\mathbb{R}}}

\newcommand*{\bbN}{\ensuremath{\mathbb{N}}}

\renewcommand{\d}{\delta}

\newcommand{\f}{\frac}

\newcommand{\ve}{\varepsilon}

\newcommand{\dd}{\hspace*{0.1em}\mathrm{d}}

\newcommand{\hf}{\frac{1}{2}}
\newcommand{\thf}{{\textstyle\frac{1}{2}}}

\begin{document}

\title[Mollification and non-vanishing on GL(3)]
        {Mollification and non-vanishing of
        \\ automorphic $L$-functions on GL(3)}
\author{Bingrong Huang}
\address{School of Mathematics \\ Shandong University \\ Jinan \\Shandong 250100 \\China}
\email{brhuang@mail.sdu.edu.cn}
\author{Shenhui Liu}
\address{231 W 18th Ave \\ MW 549 \\ Columbus, OH 43210 \\ USA}
\email{liu.2076@osu.edu}
\author{Zhao Xu}
\address{School of Mathematics \\ Shandong University \\ Jinan \\Shandong 250100 \\China}
\email{zxu@sdu.edu.cn}

\date{\today}

\begin{abstract}
  We prove a non-vanishing result for central values of $L$-functions on GL(3), by using the mollification method and the Kuznetsov trace formula.
\end{abstract}


\maketitle
\setcounter{tocdepth}{1}
\tableofcontents

\section{Introduction} \label{sec:intr}
There has been vast research on the non-vanishing of central $L$-values for families of automorphic forms, since the pioneering work of Duke \cite{duke1995} and Iwaniec--Sarnak \cite{iwaniec-sarnak1999dirichlet,iwaniec-sarnak2000siegel}. To get \textit{positive-proportional} non-vanishing results in families, one typically turns to the method of moments and the mollification method \`{a} la Selberg (see, for example, \cite{KowalskiMichel1999, VanderKam1999, KowalskiMichel2000, KMV2000, KMV2000b, Soundararajan2000, KMV2002, Blomer2008, Khan2010, Luo2015}, and others). In the current work we follow this approach and go beyond families of GL(2) forms (and symmetric-square lifts of GL(2) forms), and study the central $L$-values of Maass forms on GL(3) and prove a non-vanishing result of such values (Theorem \ref{main result}), which is a positive-proportional result in the sense of Remark \ref{remark_main_result}.


To state our result, we introduce a few notations and refer the reader to \S\,\ref{subsec:cusp} for certain details. Pick an orthogonal basis $\{\phi_j\}$ of Hecke--Maass forms for $\Gamma=SL(3,\mathbb{Z})$. Each $\phi_j$ has spectral parameter $\nu_{j}=\big(\nu_{j,1},\nu_{j,2},\nu_{j,3}\big)$, the Langlands parameter
$\mu_{j}=\big(\mu_{j,1},\mu_{j,2},\mu_{j,3}\big)$, and the Hecke eigenvalues $A_j(n,1)$. The main objects under investigation are the $L$-functions
\[
  L(s,\phi_j) = \sum_{n=1}^{\infty} \frac{A_{j}(1,n)}{n^{s}}
    \quad \textrm{for } \Re(s)>1.
\]
A simple observation is that there is no trivial reason for $L(\frac{1}{2},\phi_j)$ to vanish, since every $\phi_j$ is necessarily even and the sign of the functional equation of $L(s,\phi_j)$ is positive. In fact, one expects many of $L(\frac{1}{2},\phi_j)$ to be nonzero.
%
As in Blomer--Buttcane \cite{blomer2015subconvexity}, we consider the generic case in short interval.
Let $\mu_0=(\mu_{0,1},\mu_{0,2},\mu_{0,3})$ and $\nu_0=(\nu_{0,1},\nu_{0,2},\nu_{0,3})$, and satisfy
the corresponding relations \eqref{eqn:m2n} and \eqref{eqn:n2m}.
We also assume
\[
  |\mu_{0,i}|\asymp|\nu_{0,i}|\asymp T:=\|\mu_0\|, \quad
  1\leq i \leq 3.
\]
Let $M=T^\theta$ for any fixed $0<\theta<1$. Define a test function $h_{T,M}(\mu)$ (depending on $\mu_0$) for $\mu=(\mu_1,\mu_2,\mu_3)\in\mathbb{C}^3$ by
$$
     h_{T,M}(\mu) := P(\mu)^2 \bigg(\sum_{w\in\cW}\psi\bigg(\frac{w(\mu)-\mu_0}{M}\bigg)\bigg)^2,
$$
where
$$
    \psi(\mu) =\exp\left(\mu_{1}^2+\mu_{2}^2+\mu_{3}^2\right)
$$
and
$$
    P(\mu)=\prod_{0\leq n\leq A} \prod_{k=1}^{3} \frac{\left(\nu_{k}-\frac{1}{3}(1+2n)\right)\left(\nu_{k}+\frac{1}{3}(1+2n)\right)}{|\nu_{0,k}|^2}
$$
for some fixed large $A>0$. Here
\[
  \mathcal{W} := \left\{ I,\; w_2=\left(\begin{smallmatrix} 1 & & \\  & & 1 \\  &1&   \end{smallmatrix}\right),\;
  w_3=\left(\begin{smallmatrix}  &1& \\  1&& \\  &&1   \end{smallmatrix}\right),\;
  w_4=\left(\begin{smallmatrix}  &1& \\  &&1 \\  1&&   \end{smallmatrix}\right),\;
  w_5=\left(\begin{smallmatrix}  &&1 \\  1&& \\  &1&   \end{smallmatrix}\right),\;
  w_6=\left(\begin{smallmatrix}  &&1 \\  &1& \\  1&&   \end{smallmatrix}\right)
  \right\}
\]
is the Weyl group of $SL(3,\bbR)$. The function $h(\mu)$ has the localizing effect at a ball of radius $M$ about $w(\mu_0)$ for each $w\in\cW$, and other nice properties stated in \S\,\ref{subsec:cusp}. Then with the normalizing factor
$$
    \cN_j:=\|\phi_j\|^2\prod_{k=1}^{3}\cos\left(\frac{3}{2}\pi\nu_{j,k}\right),
$$
our main result is as follows.
\begin{theorem}\label{main result}
  We have
  \[
    \sum_{\substack{j\geq1 \\ L(\thf,\phi_j)\neq0}} \frac{h_{T,M}(\mu_j)}{\cN_j} \gg T^3M^2.
  \]
\end{theorem}

\begin{remark}\label{remark_main_result}
  By a stronger form of the GL(3) spectral large sieve inequality obtained by Young (\cite[Theorem 1.1]{young2016largesieve}), one can get the following weighted Weyl law:
  \[
    \sum_{j\geq1} \frac{h_{T,M}(\mu_j)}{\cN_j} \ll T^3M^2.
  \]
  In fact, we can replace the above ``$\ll$" by ``$\asymp$".
  Thus in this sense Theorem \ref{main result} gives a positive-proportional non-vanishing result in short interval.
\end{remark}

Next we outline the structure of the paper and give the proof of Theorem \ref{main result}. In \S\,\ref{subsec:cusp} we briefly review facts of Maass forms and their $L$-functions, as well as the main analytic tool, the GL(3) Kuznetsov trace formula (Lemma \ref{lemma: KTF}). Define the mollifier $M_j$ for $L(\hf,\phi_j)$ by
\[
  M_j := \sum_{\ell\leq L}\frac{A_j(1,\ell)}{\ell^{1/2}} x_{\ell},
\]
where
\begin{align}
  x_{\ell} := \mu(\ell)\f{1}{2\pi i}\int_{(3)}\f{(L/\ell)^s}{s^2}\frac{\dd s}{\log L}
  =  \begin{cases}
        \mu(\ell)\f{\log(L/\ell)}{\log L},\ \ &\ell\leq L,
        \\\noalign{\vskip 1mm}
        0,\ \ \ &\ell>L,
     \end{cases}
\end{align}
and $L=T^{\delta}$ for some small $\delta>0$. Then we study the mollified moments of the central $L$-values and prove the following two propositions
in \S\ref{sec:2ndmoment} and \S\ref{sec:1stmoment}, respectively.

\begin{proposition}\label{second moment}
  We have
  \[
    \sum_{j} \frac{h_{T,M}(\mu_j)}{\cN_j} |L(\thf,\phi_j) M_j|^2 \ll T^3M^2,
  \]
  provided $0<\delta<11/78$.
\end{proposition}

\begin{proposition}\label{first moment}
  We have
  \[
    \sum_{j} \frac{h_{T,M}(\mu_j)}{\cN_j} L(\thf,\phi_j) M_j \asymp T^3M^2,
  \]
  provided $0<\delta<11/78$.
\end{proposition}
\begin{remark}
  The above results can be improved. The restriction of $\delta$ comes from the contribution of Eisenstein series, which can be refined if we use the subconvexity bounds for GL(1) and GL(2) $L$-functions or the average Lindel\"{o}f bound of the related families of $L$-functions.
\end{remark}


Through out the paper, $\varepsilon$ is an arbitrarily small positive number
and $B$ is a sufficiently large positive number which may not be the same at each occurrence.

\section{Preliminaries}
In this section we review essential facts and tools, required for later development.
\subsection{Hecke--Maass cusp forms and their $L$-functions}\label{subsec:cusp}
Let $G=GL(3,\mathbb{R})$ with maximal compact subgroup $K=O(3,\mathbb{R})$ and center $Z\cong\mathbb{R}^\times$. Let $\mathbb{H}_3=G/(K\cdot Z)$ be the generalized upper half-plane.
For $0\leq c\leq \infty$, let
\begin{equation*}
  \Lambda_{c} = \Big\{ \mu=(\mu_1,\mu_2,\mu_3)\in\mathbb{C}^3 \; \Big| \; \mu_1+\mu_2+\mu_3=0,\ |\Re\mu_k|\leq c,\ k=1,2,3 \Big\},
\end{equation*}
and
\begin{equation*}
  \Lambda_{c}' = \Big\{ \mu\in\Lambda_c \; \Big| \; \{-\mu_1,-\mu_2,-\mu_3\}=\{\overline{\mu_1},\overline{\mu_2},\overline{\mu_3}\} \Big\}.
\end{equation*}
Consider a Hecke--Maass form $\phi$ in $L^2(SL(3,\mathbb{Z})\backslash\mathbb{H})$. Here $\mu$ will be the Langlands parameters.
Define the spectral parameters
\begin{equation}\label{eqn:n2m}
  \nu_1=\frac{1}{3}(\mu_1-\mu_2),\quad
  \nu_2=\frac{1}{3}(\mu_2-\mu_3),\quad
  \nu_3=\frac{1}{3}(\mu_3-\mu_1).
\end{equation}
We have
\begin{equation}\label{eqn:m2n}
  \mu_1=2\nu_1+\nu_2, \quad
  \mu_2=\nu_2-\nu_1, \quad
  \mu_3=-\nu_1-2\nu_2.
\end{equation}
We will simultaneously use $\mu$ and $\nu$.
By unitarity and the standard Jacquet--Shalika bounds, the Langlands parameter of an
arbitrary irreducible representation $\pi\subseteq L^2(SL(3,\mathbb{Z})\backslash\mathbb{H}_3)$
is contained in $\Lambda_{1/2}'\subseteq \Lambda_{1/2}$,
and the non-exceptional parameters are in $\Lambda_{0}'=\Lambda_{0}$.

Let $\phi$ be a Hecke--Maass cusp form for $\Gamma=SL(3,\mathbb{Z})$ with Fourier coefficients
$A_{\phi}(m,n)\in\mathbb{C}$ for $(m,n)\in \mathbb{N}^2$.
The standard L-function of $\phi$ is given by
\[
  L(s,\phi):=\sum_{n=1}^{\infty} \frac{A_\phi(1,n)}{n^s} \quad \textrm{for }\Re(s)>1.
\]
For the dual form $\widetilde{\phi}$ the coefficients of $L(s,\widetilde{\phi})$ are $A_\phi(1,n)=\overline{A_\phi(n,1)}$.
The functional equation of $L(s,\phi)$ is
\begin{equation}\label{eqn:FE}
  L(s,\phi)\prod_{j=1}^{3}\Gamma_{\mathbb{R}}(s-\mu_j)
  = L(1-s,\widetilde{\phi})\prod_{j=1}^{3}\Gamma_{\mathbb{R}}(1-s+\mu_j) ,
\end{equation}
where $\Gamma_{\mathbb{R}}(s) = \pi^{-s/2}\Gamma(s/2)$.

\subsection{The minimal Eisenstein series and its Fourier coefficients}

Let
$$
U_3=\left\{\begin{pmatrix} 1&*&* \\ 0&1&* \\ 0&0&1 \end{pmatrix}\right\} \cap \Gamma.
$$
For $z\in \mathbb{H}_3$ and $\Re(\nu_1),\Re(\nu_2)$ sufficiently large,
we define the minimal Eisenstein series
\[
  E(z;\mu_1,\mu_2) := \sum_{\gamma\in U_3\backslash\Gamma} I_{\nu_1,\nu_2}(\gamma z),
\]
where
\[
  I_{\nu_1,\nu_2}(z) := y_1^{1+\nu_1+2\nu_2}y_2^{1+2\nu_1+\nu_2},
\]
and
\begin{equation}\label{z}
  z=\left\{\begin{pmatrix} y_1y_2&y_1x_2&x_3 \\ 0&y_1&x_1 \\ 0&0&1 \end{pmatrix}\right\}
  =\left\{\begin{pmatrix} 1&x_2&x_3 \\ 0&1&x_1 \\ 0&0&1 \end{pmatrix}\right\}\left\{\begin{pmatrix} y_1y_2&0&0 \\ 0&y_1&0 \\ 0&0&1 \end{pmatrix}\right\},
\end{equation}
with $y_1,y_2>0$ and $x_1,x_2,x_3\in\mathbb{R}$.
It has meromorphic continuation in $\nu_1$ and $\nu_2$.
The Fourier coefficients $A_{\mu}(m_1,m_2)$ is defined by
(see Goldfeld \cite[Theorems 10.8.6]{goldfeld2006automorphic})
\[
  A_{\mu}(m,1) = \sum_{d_1d_2d_3=m} d_1^{\mu_1}d_2^{\mu_2}d_3^{\mu_3},
\]
and the symmetry and Hecke relation (see Goldfeld \cite[Theorems 6.4.11]{goldfeld2006automorphic})
\[
  \begin{split}
     A_{\mu_1,\mu_2}(m,1) & = \overline{A_{\mu_1,\mu_2}(1,m)}, \\
     A_{\mu_1,\mu_2}(m,n) & = \sum_{d|(m,n)} \mu(d) A_{\mu_1,\mu_2}(m/d,1)A_{\mu_1,\mu_2}(1,n/d).
  \end{split}
\]
Hence we have
\begin{equation}\label{eqn:Amin<<}
  |A_{\mu}(m_1,m_2)| \ll_{\varepsilon} (m_1m_2)^{\varepsilon},
  \quad \textrm{if}\
  \mu\in (i\mathbb{R})^2.
\end{equation}
In order to state the Kuznetsov trace formula in the \S 2.5, we introduce
$$
\mathcal{N}_\mu=\frac{1}{16}\prod_{k=1}^3|\zeta(1+3\nu_{k})|^2
$$
corresponding to the minimal Eisenstein series $E(z;\mu_1,\mu_2)$,
where $\mu=(\mu_1,\mu_2,\mu_3)$.
Recall that (see \cite{titchmarsh1986theory}) we have
$$
\frac{1}{\zeta(1+it)}\ll \log (1+|t|),
$$
which implies that
\begin{align}\label{N min}
  \frac{1}{\mathcal{N}_\mu}\ll \prod_{k=1}^3\log^2(1+|\nu_k|).
\end{align}

\subsection{The maximal Eisenstein series and its Fourier coefficients}

Let
$$
P_{2,1}=\left\{\begin{pmatrix} *&*&* \\ *&*&* \\ 0&0&* \end{pmatrix}\right\} \cap \Gamma.
$$
Let $\mu\in\mathbb{C}$ have sufficiently large real part,
and let $g: SL(2,\mathbb{Z})\backslash\mathbb{H}_2\rightarrow \mathbb{C}$
be a Hecke--Maass cusp form with $\|g\|=1$,
Langlands parameter $\mu_g\in i\mathbb{R}$
and Hecke eigenvalue $\lambda_g(m)$.
The maximal Eisenstein series twisted by a Maass form $g$ is defined by
\[
  E(z;\mu;g) := \sum_{\gamma \in P_{2,1}\backslash\Gamma}\det(\gamma z)^{1/2+\mu} g(\mathfrak{m}_{P_{2,1}}(\gamma z)),
\]
where $z$ is defined as in \eqref{z}, and
\[
  \mathfrak{m}_{P_{2,1}}: \mathbb{H}_3 \rightarrow \mathbb{H}_2,\quad
  \left(\begin{matrix}
          y_1y_2 & y_1x_2 & x_3 \\
           & y_1 & x_1 \\
           &   & 1
        \end{matrix}\right) \mapsto
  \left(\begin{matrix}
          y_2 & x_2  \\
           &   1
        \end{matrix}\right)
\]
is the restriction to the upper left corner.
It has a meromorphic continuation in $\mu$.
The Fourier coefficients are determined by
\[
  B_{\mu,g}(m,1) = \sum_{d_1d_2=m}\lambda_g(d_1)d_1^{\mu}d_2^{-2\mu},
\]
and the symmetry and Hecke relation as above
(see Goldfeld \cite[Proposition 10.9.3 and Theorem 6.4.11]{goldfeld2006automorphic}).
Recall that we have the following Kim--Sarnak bound for $GL(2)$ Fourier coefficients
(see Kim \cite[Appendix 2]{kim2003functoriality})
\[
  |\lambda_g(n)|\ll_{\varepsilon} n^{7/64+\varepsilon}.
\]
Hence we have
\begin{equation}\label{eqn:Amax<<}
  |B_{\mu,g}(m_1,m_2)| \ll_\varepsilon (m_1m_2)^{7/64+\varepsilon},
  \quad \textrm{if}\
  \mu\in i\mathbb{R}.
\end{equation}
We also introduce
$$
\mathcal{N}_{\mu,g}=8L(1,\textup{Ad}^2g)|L(1+3\mu,g)|^2,
$$
where $L(s,\textup{Ad}^2g)$ is the adjoint square $L$-function of $g$,
and $L(s,g)$ is the $L$-function of $g$. 
We have the lower bounds
$$
L(1,\textup{Ad}^2g)\gg (1+|\mu_g|)^{-\varepsilon},  \quad
L(1+it,g)\gg (1+|t|+|\mu_g|)^{-\varepsilon}.
$$
These lower bounds follow from \cite{hoffstein-lockhart, hoffstein-ramakrishnan, jacquet-shalika}, and \cite{gelbart-lapid-sarnak}.
Therefore, for $\mu\in i\mathbb{R}$, it follows that
\begin{align}\label{N max}
  \frac{1}{\mathcal{N}_{\mu,g}}\ll (1+|\mu|+|\mu_g|)^{\varepsilon}.
\end{align}

\subsection{The Kloosterman sums}

For $n_1$, $n_2$, $m_1$, $m_2$, $D_1$, $D_2\in\mathbb{N}$, we need the relevant Kloosterman sums
\begin{align*}
  \tilde{S}(n_1,n_2,m_1;D_1,D_2) :=\sum_{\substack{C_1\,(\mod D_1),\, C_2\,(\mod D_2)\\ (C_1,D_1)=(C_2,D_2/D_1)=1}}
  e\left(n_2\f{\bar{C_1}C_2}{D_1}+m_1\f{\bar{C_2}}{D_2/D_1}+n_1\f{C_1}{D_1}\right)
\end{align*}
for $D_1|D_2$, and
\begin{align*}
  &S(n_1,m_2,m_1,n_2;D_1,D_2)
  \\
  &:=\sum_{\substack{B_1,C_1(\mod D_1)\\ B_2,C_2(\mod D_2)\\ D_1C_2+B_1B_2+D_2C_1\equiv0\,(\mod D_1D_2)\\ (B_j,C_j,D_j)=1}}e\left(\f{n_1B_1+m_1(Y_1D_2-Z_1B_2)}{D_1}+\f{m_2B_2+n_2(Y_2D_1-Z_2B_1)}{D_2}\right),
\end{align*}
where $Y_jB_j+Z_jC_j\equiv1\,(\mod D_j)$ for $j=1,2$.

\subsection{The Kuznetsov trace formula}

We first introduce some notation.
Define the spectral measure on the hyperplane $\mu_1+\mu_2+\mu_3=0$ by
\[
  \dd_{\spec}\mu=\spec(\mu)\dd\mu,
\]
where
$$
  \spec(\mu):= \prod_{k=1}^{3}\left(3\nu_k\tan\left(\frac{3\pi}{2}\nu_k\right)\right)\quad\mbox{and}\quad \dd\mu=\dd\mu_1\dd\mu_2=\dd\mu_2\dd\mu_3=\dd\mu_1\dd\mu_3.
$$
Following \cite[Theorems 2 \& 3]{buttcane2016spectral},
we define the following integral kernels in terms of
Mellin--Barnes representations.
For $s \in \bbC$, $\mu \in \Lambda_{\infty}$ define the meromorphic function
$$
  \tilde{G}^{\pm}(s, \mu) := \frac{  \pi^{-3s}}{12288 \pi^{7/2}}\Biggl(\prod_{k=1}^3\frac{\Gamma(\frac{1}{2}(s-\mu_k))}{\Gamma(\frac{1}{2}(1-s+\mu_k))} \pm i   \prod_{k=1}^3\frac{\Gamma(\frac{1}{2}(1+s-\mu_k))}{\Gamma(\frac{1}{2}(2-s+\mu_k))} \Biggr),
$$
and for $s = (s_1, s_2) \in \bbC^2$, $\mu \in \Lambda_{\infty}$ define the meromorphic function
$$
  G(s, \mu) :=  \frac{1}{\Gamma(s_1 + s_2)} \prod_{k=1}^3 \Gamma(s_1 - \mu_k) \Gamma(s_2 + \mu_k).
$$
The latter is essentially the double Mellin transform
of the GL(3) Whittaker function.
We also define the following trigonometric functions
\begin{displaymath}
  \begin{split}
    & S^{++}(s, \mu) := \frac{1}{24 \pi^2} \prod_{k=1}^3 \cos\left(\frac{3}{2} \pi \nu_k\right),\\
    &  S^{+-}(s, \mu) :=  -\frac{1}{32 \pi^2} \frac{\cos(\frac{3}{2} \pi \nu_2)\sin(\pi(s_1 - \mu_1))\sin(\pi(s_2 + \mu_2))\sin(\pi(s_2 + \mu_3))}{\sin(\frac{3}{2} \pi \nu_1)\sin(\frac{3}{2} \pi \nu_3) \sin(\pi(s_1+s_2))}, \\
    & S^{-+}(s, \mu) :=-\frac{1}{32 \pi^2}  \frac{\cos(\frac{3}{2} \pi \nu_1)\sin(\pi(s_1 - \mu_1))\sin(\pi(s_1 - \mu_2))\sin(\pi(s_2 + \mu_3))}{\sin(\frac{3}{2} \pi \nu_2)\sin(\frac{3}{2} \pi \nu_3)\sin(\pi(s_1+s_2))}, \\
    & S^{--}(s, \mu) := \frac{1}{32 \pi^2}  \frac{\cos(\frac{3}{2} \pi \nu_3) \sin(\pi(s_1 - \mu_2))\sin(\pi(s_2 + \mu_2))}{\sin(\frac{3}{2} \pi \nu_2)\sin(\frac{3}{2} \pi \nu_1)}.
  \end{split}
\end{displaymath}
For  $y \in \Bbb{R} \setminus \{0\}$ with
$\sgn(y) = \epsilon$, 
let
$$
  K_{w_4}(y;\mu) :=  \int_{-i\infty}^{i\infty} | y|^{-s} \tilde{G}^{\epsilon}(s, \mu) \frac{\dd s}{2\pi i}.
$$
For $y = (y_1, y_2) \in (\bbR\setminus \{0\})^2$ with
$\sgn(y_1) = \epsilon_1$, $\sgn(y_2) = \epsilon_2$, let
\begin{displaymath}
  \begin{split}
     K^{\epsilon_1, \epsilon_2}_{w_6}(y; \mu)
     :=  & \int_{-i\infty}^{i\infty} \int_{-i\infty}^{i\infty}  |4\pi^2 y_1|^{-s_1} |4\pi^2 y_2|^{-s_2}  G(s, \mu) S^{\epsilon_1, \epsilon_2}(s, \mu)\frac{\dd s_1 \dd s_2}{(2\pi i)^2}.
  \end{split}
\end{displaymath}

We can now state the Kuznetsov trace formula in the version of
Buttcane \cite[Theorems 2, 3, 4]{buttcane2016spectral}.

\begin{lemma}\label{lemma: KTF}
  Let $n_1$, $n_2$, $m_1$, $m_2 \in \bbN$ and
  let $h$ be a function  that is holomorphic on
  $\Lambda_{1/2 + \delta}$ for some $\delta > 0$,
  symmetric under the Weyl group $\cW$, of rapid decay when
  $|\Im \mu_j| \rightarrow \infty$,  and satisfies
  \begin{equation}\label{eqn: zeros}
    h(3\nu_j \pm 1)  = 0, \quad j = 1, 2, 3.
  \end{equation}
  Then we have
  \begin{displaymath}
    \begin{split}
      &  \cC + \cE_{\min} + \cE_{\max}
       = \Delta + \Sigma_4 + \Sigma_5 + \Sigma_6,
    \end{split}
  \end{displaymath}
  where
  \begin{equation*}
    \begin{split}
       \cC & := \sum_{j} \frac{h(\mu_{j})}{\cN_j} \overline{A_{j}(m_1, m_2)} A_{j}(n_1, n_2), \\
        \cE_{\min} & := \frac{1}{24(2\pi i)^2} \iint_{\Re(\mu)=0} \frac{h(\mu)}{\cN_{\mu}}
                        \overline{A_{\mu}(m_1, m_2)} A_{\mu}(n_1, n_2) \dd\mu_1 \dd\mu_2,\\
       \cE_{\max} & :=  \sum_{g} \frac{1}{2\pi i} \int_{\Re(\mu)=0}
            \frac{h(\mu+\mu_g,\mu-\mu_g,-2\mu)}{\cN_{\mu,g}}
             \overline{B_{\mu,g}(m_1, m_2)} B_{\mu,g}(n_1, n_2) \dd\mu,\\
    \end{split}
  \end{equation*}
  and
  \begin{equation*}
    \begin{split}
      \Delta& := \delta_{n_1, m_1} \delta_{n_2, m_2}  \frac{1}{192\pi^5} \int_{\Re \mu = 0} h(\mu) \dd_{\spec}\mu,\\
       \Sigma_{4}& := \sum_{\epsilon  = \pm 1} \sum_{\substack{D_2 \mid D_1\\  m_2 D_1= n_1 D_2^2}}\frac{ \tilde{S}(-\epsilon n_2, m_2, m_1; D_2, D_1)}{D_1D_2} \Phi_{w_4}\!\left(  \frac{\epsilon m_1m_2n_2}{D_1 D_2} \right),  \\
       \Sigma_{5} &:= \sum_{\epsilon  = \pm 1} \sum_{\substack{     D_1 \mid D_2\\ m_1 D_2 = n_2 D_1^2}} \frac{ \tilde{S}(\epsilon n_1, m_1, m_2; D_1, D_2) }{D_1D_2}\Phi_{w_5}\!\left( \frac{\epsilon n_1m_1m_2}{D_1 D_2}\right),\\
       \Sigma_6 &:= \sum_{\epsilon_1, \epsilon_2 = \pm 1} \sum_{D_1,  D_2  } \frac{S(\epsilon_2 n_2, \epsilon_1 n_1, m_1, m_2; D_1, D_2)}{D_1D_2} \Phi_{w_6}\!  \left( - \frac{\epsilon_2 m_1n_2D_2}{D_1^2}, - \frac{\epsilon_1 m_2n_1D_1}{ D_2^2}\right),
    \end{split}
  \end{equation*}
  with
\begin{equation}\label{defPhi}
    \begin{split}
      & \Phi_{w_4}(y) :=  \int_{\Re \mu = 0} h(\mu) K_{w_4}(y; \mu ) \dd_{\spec}\mu,\\
      & \Phi_{w_5}(y) := \int_{\Re \mu = 0} h(\mu) K_{w_4}(-y; -\mu ) \dd_{\spec}\mu,\\
      & \Phi_{w_6}(y_1, y_2) := \int_{\Re \mu = 0} h(\mu) K^{\sgn(y_1), \sgn(y_2)}_{w_6}((y_1, y_2) ;  \mu ) \dd_{\spec}\mu.
    \end{split}
\end{equation}
\end{lemma}

In the first moment, we will use $h(\mu)=h_{T,M}(\mu)$ to be the test function which is same as Blomer--Buttcane \cite{blomer2015subconvexity}, and in the second moment, we will use
$h_2(\mu)=h_{T,M}(\mu)W_{\mu,N}(m_1m_2)$ (see \eqref{test h2}) to be the test function.
The function $h(\mu)$ localizes at a ball of radius $M$ about $w(\mu_0)$ for each $w\in\mathcal{W}$.
We have
\begin{equation}\label{property1 for h}
\mathscr{D}_k h_{T,M}(\mu)\ll_k M^{-k},
\end{equation}
for any differential operator $\mathscr{D}_k$ of order $k$,
which we use frequently when we integrate by parts,
and sufficiently many differentiations can save arbitrarily many powers of $T$.
Moreover, by trivial estimate, we have
\begin{equation}\label{property2 for h}
  \int_{\Re \mu=0}h_{T,M}(\mu)\dd_{\spec}\mu\asymp T^3M^2.
\end{equation}

\subsection{The weight functions}

For the weight functions, we will need the following results
in Blomer--Buttcane \cite[Lemma 1, Lemma 8, and Lemma 9]{blomer2015subconvexity}.

\begin{lemma}\label{truncate 1st}
  For some large enough constant $B>0$, we have
  \[
    \Phi_{w_4}(y) \ll |y|^{1/10}T^B,\quad
    \Phi_{w_5}(y) \ll |y|^{1/10}T^B,\quad
    \Phi_{w_6}(y_1,y_2) \ll |y_1y_2|^{3/5}T^B.
  \]
\end{lemma}


\begin{lemma}\label{truncate 2nd}
  \begin{itemize}
    \item [(i)]    If $0<|y|\leq T^{3-\ve}$, then for any constant $B>0$, we have
                    \[
                      \Phi_{w_4}(y) \ll_{\ve,B} T^{-B}.
                    \]
    \item [(ii)]   If $T^{3-\varepsilon}<y$, then
                    \[
                      y^{k}\Phi_{w_4}^{(k)}(y)\ll_{k,\varepsilon} T^{3+2\varepsilon}M^2(T+|y|^{1/3})^{k}.
                    \]
  \end{itemize}
\end{lemma}


\begin{lemma}\label{truncate 3rd}
  Let $\Upsilon:=\min\{|y_1|^{1/3}|y_2|^{1/6},|y_1|^{1/6}|y_2|^{1/3}\}$.
  If $\Upsilon\ll T^{1-\ve}$, then we have
  \[ \Phi_{w_6}(y_1,y_2) \ll T^{-B}. \]
\end{lemma}

There is a slight difference that we use \eqref{property1 for h} and \eqref{property2 for h}
instead of \cite[(3.7) and (3.8)]{blomer2015subconvexity}.
This have no influence in the proof.
We define $\Phi_{2,w_4}(y)$, $\Phi_{2,w_5}(y)$, and $\Phi_{2,w_6}(y)$ as in \eqref{defPhi} by using
the test function $h_2(\mu)$.
In the proof of the above three Lemmas, the only two properties of $h_{T,M}(\mu)$ which are used is \eqref{property1 for h} and \eqref{property2 for h}.
Here, we remark that $h_2(\mu)$ also satisfies these two inequalities
by using properties of $W_{\mu,N}$ (see \S \ref{sec:2ndmoment}).
So, $\Phi_{2,w_4}(y)$, $\Phi_{2,w_5}(y)$, and $\Phi_{2,w_6}(y)$ also satisfies the corresponding bounds in the above three Lemmas.

\section{The mollified second moment}\label{sec:2ndmoment}
Let $G(s)=(\cos\f{\pi s}{A})^{-100A}$, where $A$ is a positive integer.
For $|L(\thf,\phi_j)|^2$, we will use the
following approximate functional equation.

\begin{lemma}\label{lemma:AFE2}
  Let $\phi_j$ be a $GL(3)$ Maass form with Langlands parameters $\mu_j \in \Lambda_{1/2}'$.
  We have
  \[
    |L(\thf,\phi_j)|^2 = 2 \sum_{m_1,m_2}
    \frac{\overline{A_j(m_1,m_2)}}{(m_1m_2)^{1/2}}W_{j}\left(m_1m_2\right),
  \]
  where
  \[
    W_{j}(y):= \frac{1}{2\pi i}\int_{(3)} \zeta(1+2s) (\pi^3y)^{-s} \prod_{k=1}^{3}\prod_\pm \frac{\Gamma\left(\frac{s+1/2\pm\mu_{j,k}}{2}\right)}{\Gamma\left(\frac{1/2\pm\mu_{j,k}}{2}\right)}
    G(s)\frac{\dd s}{s}.
  \]
  Moreover, we have
  \[
    y^{i}W_{j}^{(i)}(y) \ll_{i,B} \bigg(1+\frac{y}{\prod_{k=1}^{3}(1+|\mu_{j,k}|)}\bigg)^B,
  \]
  for any non-negative integer $i$, and any large positive integer $B$.
\end{lemma}

\begin{proof}
  See e.g. Iwaniec--Kowalski \cite[\S5.2]{iwaniec2004analytic}.
\end{proof}

Note that the sum in Proposition \ref{second moment} is essentially supported on the generic forms
which satisfy
\[
  |\mu_{j,k}|\asymp|\nu_{j,k}|\asymp T, \quad 1\leq k \leq 3.
\]
So we assume $\phi_j$ also satisfy the above relation.
By Stirling's formula, if $|s|\ll |z|^{1/2}$, then
$$
\f{\Gamma(z+s)}{\Gamma(z)}=z^s\left(1+\sum_{n=1}^N\f{P_n(s)}{z^{n}}+O\Big(\f{(1+|s|)^{2N+2}}{|z|^{N+1}}\Big)\right),
$$
for certain polynomials $P_n(s)$ of degree $2n$. Since $G(s)$ has exponential decay, we may truncate at
$|\Im s|\ll T^\varepsilon$ with only a small error in $W_{j}(y)$.
Based on the above arguments, together with $\mu_j \in \Lambda_{1/2}'$, we have
\begin{align*}
    & W_{j}(y) = \frac{1}{2\pi i} \int_{(3)} \zeta(1+2s) (\pi^3y)^{-s} \prod_{k=1}^{3} \prod_\pm \Big(\frac{1/2\pm\mu_{j,k}}{2}\Big)^{s/2} \\
    & \hskip 80pt \cdot \left(1+\sum_{n=1}^N\f{2^nP_n(s)}{(1/2\pm\mu_{j,k})^{n}}
    +O\Big(\f{(1+|s|)^{2N+2}}{|\mu_k|^{N+1}}\Big)\right)
    G(s)\frac{\dd s}{s}.
\end{align*}
On the other hand, by Hecke multiplicativity relations we have
\[
  \begin{split}
    |M_j|^2 & = \sum_{\ell_1}\sum_{\ell_2} \frac{1}{\ell_1^{1/2}\ell_2^{1/2}}
      x_{\ell_1}\overline{x_{\ell_2}} A_j(\ell_1,1)\overline{A_j(\ell_2,1)} \\
    & =\sum_{d}\sum_{\ell_1}\sum_{\ell_2} \frac{1}{d\ell_1^{1/2}\ell_2^{1/2}}
        x_{d\ell_1}\overline{x_{d\ell_2}} A_j(\ell_1,\ell_2),
  \end{split}
\]
By Lemma \ref{lemma:AFE2} and Hecke relations again, we have
\[
  \begin{split}
    & \sum_{j} \frac{h_{T,M}(\mu_j)}{\cN_j} |L(\thf,\phi_j)M_j|^2 \\
    & \hskip 30pt = 2 \sum_{j} \frac{h_{T,M}(\mu_j)}{\cN_j} \sum_{d}\sum_{\ell_1}\sum_{\ell_2} \frac{x_{d\ell_1}\overline{x_{d\ell_2}}}{d\ell_1^{1/2}\ell_2^{1/2}}
         A_j(\ell_1,\ell_2)
        \sum_{m_1,m_2}\frac{1}{m_1^{1/2}m_2^{1/2}}
        W_{j}\left(m_1m_2\right)\overline{A_j(m_1,m_2)} \\
    & \hskip 30pt = 2 \sum_{d}\frac{1}{d}\sum_{\ell_1,\ell_2}
        \frac{x_{d\ell_1}\overline{x_{d\ell_2}}}{\ell_1^{1/2}\ell_2^{1/2}}
        \sum_{m_1,m_2}\frac{1}{m_1^{1/2}m_2^{1/2}}
        \sum_{j}\frac{h_{T,M}(\mu_j) W_{j}\left(m_1m_2\right)}{\cN_j}
        A_j(\ell_1,\ell_2) \overline{A_j(m_1,m_2)}.
  \end{split}
\]
Let
$$
  W_{j,N}(y) = \frac{1}{2\pi i} \int_{(3)} \zeta(1+2s) (\pi^3y)^{-s}\prod_{k=1}^{3}\prod_\pm \left(\frac{1/2\pm\mu_{j,k}}{2}\right)^{s/2}
  \left(1+\sum_{n=1}^N\f{2^nP_n(s)}{(1/2\pm\mu_{j,k})^{n}}\right)
    G(s)\frac{\dd s}{s}.
$$
Then we can replace $W_{j}$ by $W_{j,N}$ with a negligible error if we choose $N$ to be large enough.
Now we use the Kuznetsov trace formula (see Lemma \ref{lemma: KTF}) with the test function
\begin{align}\label{test h2}
h_2(\mu)=h_{T,M}(\mu)W_{\mu,N}(m_1m_2),
\end{align}
where $W_{\mu,N}$ is defined the same as $W_{j,N}$ with $\mu$ replacing $\mu_j$.
It turns out we are led to estimate
\begin{equation}\label{use Ku 2}
  \begin{split}
      2 \sum_{d}\frac{1}{d}\sum_{\ell_1,\ell_2}
        \frac{x_{d\ell_1}\overline{x_{d\ell_2}}}{\ell_1^{1/2}\ell_2^{1/2}}
        \sum_{m_1,m_2}\frac{1}{m_1^{1/2}m_2^{1/2}}
        \Big(\Delta^{(2)}+\Sigma_4^{(2)}+\Sigma_5^{(2)}+\Sigma_6^{(2)}
        -\mathcal{E}_{\mathrm{max}}^{(2)}-\mathcal{E}_{\mathrm{min}}^{(2)}\Big),
  \end{split}
\end{equation}
where
\begin{displaymath}
    \begin{split}
      \Delta^{(2)} & := \delta_{\ell_1, m_1} \delta_{\ell_2, m_2}
            \frac{1}{192\pi^5} \int_{\Re \mu = 0} h_2(\mu) \dd_{\spec}\mu,\\
      \Sigma_{4}^{(2)}& := \sum_{\epsilon  = \pm 1} \sum_{\substack{D_2 \mid D_1\\  m_2 D_1= \ell_1 D_2^2}}
            \frac{\tilde{S}(-\epsilon \ell_2, m_2, m_1; D_2, D_1)}{D_1D_2}
            \Phi_{2,w_4}\!\left(  \frac{\epsilon m_1m_2\ell_2}{D_1 D_2} \right),  \\
      \Sigma_{5}^{(2)} & := \sum_{\epsilon  = \pm 1} \sum_{\substack{D_1 \mid D_2\\ m_1 D_2 = \ell_2 D_1^2}}
            \frac{\tilde{S}(\epsilon \ell_1, m_1, m_2; D_1, D_2) }{D_1D_2}
            \Phi_{2,w_5}\!\left( \frac{\epsilon \ell_1m_1m_2}{D_1 D_2}\right),\\
      \Sigma_6^{(2)} & := \sum_{\epsilon_1, \epsilon_2 = \pm 1} \sum_{D_1,  D_2  }
            \frac{S(\epsilon_2 \ell_2, \epsilon_1 \ell_1, m_1, m_2; D_1, D_2)}{D_1D_2}
            \Phi_{2,w_6}\!  \left( - \frac{\epsilon_2 m_1\ell_2D_2}{D_1^2},
            - \frac{\epsilon_1 m_2\ell_1D_1}{ D_2^2}\right),
    \end{split}
\end{displaymath}
with $\Phi_{2,w_4}(y)$, $\Phi_{2,w_5}(y)$, and $\Phi_{2,w_6}(y)$ defined as in \eqref{defPhi} by using
the new test function $h_2(\mu)$ given by \eqref{test h2}, respectively;
and
\begin{equation*}
    \begin{split}
       \cE_{\max}^{(2)} & :=  \sum_{g} \frac{1}{2\pi i} \int\limits_{\Re(\mu)=0}
            \frac{h_2(\mu+\mu_g,\mu-\mu_g,-2\mu)}{\cN_{\mu,g}}
             \overline{B_{\mu,g}(m_1, m_2)} B_{\mu,g}(\ell_1, \ell_2) \dd\mu, \\
       \cE_{\min}^{(2)} & := \frac{1}{24(2\pi i)^2} \iint\limits_{\Re(\mu)=0} \frac{h_2(\mu)}{\cN_{\mu}}
                        \overline{A_{\mu}(m_1, m_2)} A_{\mu}(\ell_1, \ell_2) \dd\mu_1 \dd\mu_2.
    \end{split}
\end{equation*}

\subsection{The diagonal term}

Note that we have $\ell_1=m_1$, and $\ell_2=m_2$.
Thus we infer that the diagonal term in \eqref{use Ku 2} is
\begin{align}\label{diagonal}
  \f{1}{96\pi^5}\sum_{d}\sum_{\ell_1}\sum_{\ell_2}\f{x_{d\ell_1}x_{d\ell_2}}{\ell_1\ell_2}
  \int_{\Re \mu = 0} h_{T,M}(\mu) W_{\mu,N}\left(\ell_1\ell_2\right) \dd_{\spec}\mu.
\end{align}
By the definition of $W_{\mu,N}$, we only need to deal with the leading term,
which contributes
\begin{align*}
  \f{1}{96\pi^5}\int_{\Re \mu = 0} h_{T,M}(\mu) S(\mu) \dd_{\spec}\mu,
\end{align*}
where
\begin{align*}
  S(\mu) :=\sum_{d} \f{\mu^2(d)}{d} \sum_{(\ell_1,d)=1}\sum_{(\ell_2,d)=1} \f{\mu(\ell_1)a_{d\ell_1}\mu(\ell_2)a_{d\ell_2}}{\ell_1\ell_2}
  \mathcal{W}_{\mu}(\ell_1\ell_2),
\end{align*}
with
\begin{align*}
  a_{\ell} & := \f{1}{2\pi i}\int_{(3)}\f{(L/\ell)^s}{s^2}\frac{\dd s}{\log L}
  =  \begin{cases}
        \f{\log(L/\ell)}{\log L},\ \ &\ell\leq L,
        \\\noalign{\vskip 1mm}
        0,\ \ \ &\ell>L,
     \end{cases} \\
  \mathcal{W}_{\mu}(y) & := \f{1}{2\pi i} \int_{(3)} \zeta(1+2s)\left(\f{\pi^3 y}{|\mu_1\mu_2\mu_3|}\right)^{-s}G(s)\f{\dd s}{s}.
\end{align*}
By moving the line of integration in $\mathcal{W}_{\mu}$ to left,
we pass the double pole at $s=0$. Hence, by the residue theorem, we infer that
\begin{align*}
  \mathcal{W}_{\mu}(\ell_1\ell_2)=c_1\log|\mu_1\mu_2\mu_3|+c_2\log \ell_1\ell_2+c_3+O(T^{-B}),
\end{align*}
where $c_1$, $c_2$, and $c_3$ are constants. The main contribution from $\mathcal{W}_{\mu}(\ell_1\ell_2)$
comes from the first two terms, which can be treated similarly. For convenience, we only give the details
of the first term. Note that the goal is to show
\begin{align}\label{goal}
  \sum_{d}\f{\mu^2(d)}{d} \sum_{(\ell_1,d)=1} \sum_{(\ell_2,d)=1} \f{\mu(\ell_1)a_{d\ell_1}\mu(\ell_2)a_{d\ell_2}}{\ell_1\ell_2}
  \int_{\Re \mu=0}h_{T,M}(\mu) \log|\mu_1\mu_2\mu_3|\dd_{\spec}\mu
  \ll T^{3}M^2.
\end{align}
For the $\ell_1$-sum, we have
\begin{align}\label{l1-sum}
  \sum_{(\ell_1,d)=1}\f{\mu(\ell_1)a_{d\ell_1}}{\ell_1}
  =\frac{1}{2\pi i} \int_{(2)}\prod_{p|d}\left(1-\f{1}{p^{1+s}}\right)^{-1}
  \f{(L/(d\ell_1))^s}{s^2\zeta(1+s)\log L}\dd s.
\end{align}
Recall that in the region $\Re s\geq1-c/\log(|\Im s|+3)$ (here $c$ is some positive constant),
$\zeta(s)$ is analytic except for a single pole at $s=1$, and has no zeros and satisfies
$\zeta(s)^{-1}\ll \log(|\Im s|+3)$, $\zeta'(s)/\zeta(s)\ll \log(|\Im s|+3)$
(see e.g. \cite[(3.11.7) and (3.11.8)]{titchmarsh1986theory}).
We move the line of integration in \eqref{l1-sum} to
\begin{equation}\label{defGamma}
  \mathcal{C}_{\varepsilon}
  := \big\{{\rm i}x : |x|\geq\varepsilon\big\}\cup
  \big\{\varepsilon \text{e}^{{\rm i}\vartheta} : \tfrac{\pi}{2}\leq\vartheta\leq\tfrac{3\pi}{2}\big\},
\end{equation}
and $\varepsilon$ is sufficiently small.
It follows that
\begin{align}\label{l1-moving}
  \eqref{l1-sum} = \f{1}{\log L} \prod_{p\mid d} \frac{p}{p-1}
  +\int_{\mathcal{C}_{\varepsilon}}\prod_{p|d}\left(1-\f{1}{p^{1+s}}\right)^{-1}
  \f{(L/(d\ell_1))^s}{s^2\zeta(1+s)\log L}\dd s.
\end{align}
We have the similar expression for the $\ell_2$-sum.
Inserting these into \eqref{goal}, we consider the resulting $d$-sum.
A typical term is
\begin{align}\label{typical term}
\sum_{d\leq L}\f{\mu^2(d)}{d}\prod_{p\mid d} \left(\frac{p}{p-1}\right)^2\ll \log T
\end{align}
which implies \eqref{goal} by trivial computation.
For the term involving $\log(\ell_1\ell_2)=\log \ell_1+\log \ell_2$,
we have
\begin{align*}
  \sum_{\substack{\ell_1\geq1\\(\ell_1,d)=1}}\f{\mu(\ell_1)\log \ell_1}{\ell_1^{1+s}}
  = - \bigg\{\f{1}{\zeta(1+s)}\prod_{p\mid d} \bigg(1-\frac{1}{p^{s+1}}\bigg)^{-1}\bigg\}'.
\end{align*}
A similar argument shows that its contribution to \eqref{diagonal} is $\ll T^3M^2$.

\subsection{The $w_4$ and $w_5$ terms}

We only deal with the $w_4$-term, since the $w_5$-term is very similar.
Inserting a smooth unity to $m_1$, $m_2$ sums, we are led to estimate
\begin{align}\label{w4 contribution}
\sum_{\epsilon=\pm1}\sum_{d}\frac{1}{d}\sum_{\ell_1,\ell_2}
        \frac{1}{\sqrt{\ell_1\ell_2}}
        \Sigma_4^{(2)}(\epsilon,d,\ell_1,\ell_2,M_1,M_2),
\end{align}
where $M_1,M_2\gg1$, $M_1M_2\ll T^{3+\varepsilon}$,
\begin{equation*}
  \begin{split}
    & \Sigma_4^{(2)}(\epsilon,d,\ell_1,\ell_2,M_1,M_2) \\
    & \hskip 30pt :=\sum_{m_1=1}^\infty\sum_{m_2=1}^\infty \f{W_1(\f{m_1}{M_1})W_2(\f{m_2}{M_2})}{\sqrt{m_1m_2}}
      \sum_{\substack{D,\delta\\ m_2\d=\ell_1D}}\frac{\tilde{S}(-\epsilon \ell_2, m_2, m_1; D, D\d)}{D^2\d}
            \Phi_{2,w_4}\!\left(  \frac{\epsilon m_1m_2\ell_2}{D^2\d} \right),
  \end{split}
\end{equation*}
and $W_i(x)$ ($i=1,2$) are compactly supported in $[1,2]$ and satisfy $x^jW_i^{(j)}(x)\ll 1$.
By Lemma \ref{truncate 1st}, we can truncate the $D$, $\d$ sums
at some $T^B$ for some larege $B$ at the cost of a negligible error.
Then by Lemma \ref{truncate 2nd}, we can truncate the sums, again with a negligible error, at
\begin{align}\label{eqn:condition1}
  \f{m_1m_2\ell_2}{D^2\d}\geq T^{3-\varepsilon},
\end{align}
or in other words,
$$
D^2\d\leq \f{M_1M_2L}{T^{3-\varepsilon}}\ll LT^\varepsilon.
$$
Note that we have $m_2\d=\ell_1D$ now, which implies $M_2$ is small. That is,
$$
  M_2 \leq LD/\d \leq (M_1M_2/T^{3-\varepsilon})^{1/2}L^{3/2}/\d^{3/2}
  \ll L^{3/2}T^\varepsilon/\d^{3/2}.
$$
And by \eqref{eqn:condition1} we have
$$
  M_1\gg T^{3-\varepsilon}D^2\delta/(m_2\ell_2)
  = T^{3-\varepsilon}D\delta^2/(\ell_1\ell_2)\gg T^{3-\varepsilon}/L^2.
$$
We apply the Poisson summation formula in the $m_1$-variable, getting
\begin{align*}
  & \sum_{m_1=1}^\infty \f{W_1(\f{m_1}{M_1})}{\sqrt{m_1}}\tilde{S}(-\epsilon \ell_2, m_2, m_1; D, D\d)
  \Phi_{2,w_4}\!\left(  \frac{\epsilon m_1m_2\ell_2}{D^2\d} \right)
  \\
  & \hskip 50pt =\sum_{a(\mod \d)}\tilde{S}(-\epsilon \ell_2, m_2, a; D, D\d)\sum_{m_1\equiv a(\mod \d)}
  \f{W_1(\f{m_1}{M_1})}{\sqrt{m_1}}\Phi_{2,w_4}\!\left(  \frac{\epsilon m_1m_2\ell_2}{D^2\d} \right)
  \\
  & \hskip 50pt =\f{M_1^{1/2}}{\d}\sum_{a(\mod \d)}\tilde{S}(-\epsilon \ell_2, m_2, a; D, D\d)
  \sum_{m\in\mathbb{Z}}e\left(\f{am}{\d}\right)\\
  & \hskip 150pt  \cdot
  \int_{-\infty}^\infty \f{W_1(x)}{x^{1/2}}
  \Phi_{2,w_4}\!\left(\frac{\epsilon m_2\ell_2M_1x}{D^2\d} \right)e\left(-\f{mM_1x}{\d}\right)\dd x.
\end{align*}
Integration by parts in connection with Lemma \ref{truncate 2nd} ${\rm (ii)}$ and the above bounds for $M_1$ and $M_2$
shows that the integral is negligible unless $m=0$.
When $m=0$, by opening $\tilde{S}(-\epsilon \ell_2, m_2, a; D, D\d)$ and compute the $a$ sum, we obtain
\begin{align*}
\sum_{a(\mod \d)}\tilde{S}(-\epsilon \ell_2, m_2, a; D, D\d)=\sum_{a(\mod \d)}\sum_{\substack {C_1(\mod D), C_2(\mod D\d)\\(C_1,D)=(C_2,\d)=1}}e\left(m_2\f{\bar{C_1}C_2}{D}+a\f{\bar{C_2}}{\d}-\epsilon \ell_2\f{C_1}{D}\right)
=0，
\end{align*}
unless $\d=1$.
With the help of this and $m_2=\ell_1D$, we see that the contribution from $m=0$ to
$\Sigma_4^{(2)}(\epsilon,d,\ell_1,\ell_2,M_1,M_2)$ is
\begin{align*}
M_1^{1/2}\ell_1^{-1/2}\sum_{D=1}^\infty W_2\left(\f{\ell_1D}{M_2}\right)\f{1}{D^{3/2}}\sum_{\substack{C_1(\mod D)\\(C_1,D)=1}}
e\left(-\f{\epsilon \ell_2C_1}{D}\right)
\int_{-\infty}^\infty \f{W_1(x)}{x^{1/2}}
\Phi_{2,w_4}\!\left(\frac{\epsilon \ell_1\ell_2M_1x}{D} \right)\dd x.
\end{align*}
By inserting the definitions of $\Phi_{2,w_4}(x)$ and $K_{w_4}(y;\mu)$, the above $x$-integral becomes
\begin{align}\label{inserting Phi}
  &\int_{-\infty}^\infty \f{W_1(x)}{x^{1/2}}\int_{\Re \mu=0}h_2(\mu)K_{w_4}\left(\frac{\epsilon \ell_1\ell_2M_1x}{D};\mu\right)\dd_{\spec}\mu \dd x\nonumber
\\
  &\hskip 30pt =\int_{\Re \mu=0}h_2(\mu)\int_{-i\infty}^{i\infty}\left(\int_{-\infty}^\infty x^{-s-1/2}W_1(x)\dd x\right)
\left|\frac{\ell_1\ell_2M_1}{D}\right|^{-s}\tilde{G}^\epsilon(s,\mu)\f{ds}{2\pi i}\dd_{\spec}\mu\nonumber
\\
  &\hskip 30pt =\int_{\Re \mu=0}h_2(\mu)\int_{-i\infty}^{i\infty}\left|\frac{\ell_1\ell_2M_1}{D}\right|^{-s}
\hat{W}_1\left(\f{1}{2}-s\right)\tilde{G}^\epsilon(s,\mu)\f{\dd s}{2\pi i}\dd_{\spec}\mu,
\end{align}
where the Mellin transform $\hat{W}_1$ is entire and rapidly decaying, which means that
we can restrict the $s$-integral to $|\Im s|\leq T^\varepsilon$.
Inserting the definition of $\tilde{G}^\epsilon(s,\mu)$ into \eqref{inserting Phi}, we obtain
the corresponding $\mu$-integral:
\begin{align*}
\int_{\Re \mu=0}h(\mu)\Biggl(\prod_{j=1}^3\frac{\Gamma(\frac{1}{2}(s-\mu_j))}{\Gamma(\frac{1}{2}(1-s+\mu_j))} \pm i   \prod_{j=1}^3\frac{\Gamma(\frac{1}{2}(1+s-\mu_j))}{\Gamma(\frac{1}{2}(2-s+\mu_j))} \Biggr)\dd_{\spec}\mu.
\end{align*}
We only need to consider the first part, since the second part is very similar.
For fixed $\sigma\in\mathbb{R}$, $t\in\mathbb{R}$ and $|t|\geq 1$, we have the Stirling formula
\begin{align*}
  \Gamma(\sigma+it)=\sqrt{2\pi}e^{\pi|t|/2}|t|^{\sigma-1/2}e^{it(\log|t|-1)} e^{i(\sigma-1/2)\lambda\pi/2} \left(1+O(|t|^{-1})\right),
\end{align*}
where $\lambda=1$, if $t\geq1$, and $\lambda=-1$, if $t\leq -1$.
For convenience, we denote $\mu_1=it_1$, $\mu_2=it_2$, $\mu_3=-i(t_1+t_2)=it_3$, and $\Im s=t$.
Here, in the essential integrated range, we have $|t_k|\asymp|t_k-t_k'|\asymp T$ for $1\leq k, k'\leq 3$ and $k\neq k'$,
since we are considering the generic case.
Without loss of generality, we assume $t_1>0$, $t_2>0$, and $t_3<0$, and get
\begin{equation}\label{use stirling}
  \begin{split}
     & \quad\ \prod_{j=1}^3\frac{\Gamma(\frac{1}{2}(s-\mu_j))}{\Gamma(\frac{1}{2}(1-s+\mu_j))} \\
     & = 2^{-3/2}e^{i(-3t(\log 2+1)+\pi/2)}
       \f{e^{i(t-t_1)\log(t_1-t)+i(t-t_2)\log(t_2-t)+i(t_1+t_2+t)\log(t_1+t_2+t)}}
       {(t_1-t)^{1/2}(t_2-t)^{1/2}(t_1+t_2+t)^{1/2}}
       + O(T^{-5/2}).
  \end{split}
\end{equation}
By trivial estimate, the error term from the above to \eqref{w4 contribution}
is $O(T^{2+\varepsilon}M^2L)$.

For the main term in \eqref{use stirling}, we use partial integration to prove
its contribution is small. Actually, for the $\mu_1$-integral, we denote
the phase function
$$
\phi(t_1):=(t-t_1)\log(t_1-t)+(t_1+t_2+t)\log (t_1+t_2+t).
$$
We see that
$$
\phi'(t_1)=\log\f{t_1+t_2+t}{t_1-t}\gg 1.
$$
And hence, by partial integration many times, the integral is negligible.
We finally prove that the contribution from the $w_4$-term is $O(T^{2+\varepsilon}M^2L)$.


\subsection{The $w_6$ term}

By Lemma \ref{truncate 1st}, we can truncate the $D_1$, $D_2$ sums at some $T^B$ for some large $B$ at the cost of a negligible error.
Then by Lemma \ref{truncate 3rd}, we can truncate the sum further at
\[
  \f{(m_1\ell_2)^{1/3}(m_2\ell_1)^{1/6}}{D_1^{1/2}}\geq T^{1-\varepsilon},
  \quad \textrm{and} \quad
  \f{(m_1\ell_2)^{1/6}(m_2\ell_1)^{1/3}}{D_2^{1/2}}\geq T^{1-\varepsilon},
\]
which means
$$
\f{(m_1m_2\ell_1\ell_2)^{1/2}}{(D_1D_2)^{1/2}}\geq T^{2-\varepsilon}.
$$
This gives us
$$
D_1D_2\ll \f{m_1m_2\ell_1\ell_2}{T^{4-\varepsilon}},
$$
which is impossible provided that $L\leq T^{1/2-\varepsilon}$, since the essential sums of $m_1$ and $m_2$
are truncated at $m_1m_2\leq T^{3+\varepsilon}$.
Thus, the contribution of the $w_6$ term is $O(T^{-B})$.

\subsection{The contribution from the Eisentein series}

We only treat the contribution of the maximal Eisenstein series,
since the minimal Eisenstein series can
be handled similarly and the contribution will be smaller.
We have the Weyl law on $GL_2$ (see \cite{hejhal}),
$$
\sharp\left\{g: t_g\leq T\right\}=\frac{T^2}{12}-\frac{T\log T}{2\pi}+C_0T+O\left(\frac{T}{\log T}\right),
$$
where $C_0$ is a constant.
Combining this together with definitions of $h(\mu)$ and $h_2(\mu)$, and note that
we are considering the generic case, we see that the essential region of the integration
on $\mu$ is of length $\asymp M$, and the essential number of the sum of $\mu_g$ is of
size $\asymp TM$. Hence, by the bound \eqref{eqn:Amax<<}, \eqref{N max} and the above argument,
we have
\[
  \mathcal{E}_{\mathrm{max}}^{(2)} \ll_{\varepsilon}
  TM^2  (T^3L^2)^{7/64} T^{\varepsilon}.
\]
Hence the contribution of the right hand side of \eqref{use Ku 2}
of the maximal Eisenstein series is
\[
  \ll_{\varepsilon} L T^{3/2} TM^2  (T^3L^2)^{7/64} T^{\varepsilon}
  \ll_{\varepsilon} T^{3}M^2 T^{-11/64+\varepsilon}L^{39/32}
  \ll_{\varepsilon} T^{3}M^2,
\]
provided $L\ll T^{11/78-\varepsilon}$.

\section{The mollified first moment}\label{sec:1stmoment}

Let $G(s)$ be defined as in \S\ref{sec:2ndmoment}, and $T_0=T^{1+\varepsilon}$.
Consider the integral
\begin{align*}
  \f{1}{2\pi i}\int_{(3)}G(s)L\left(s+\f{1}{2},\phi_j\right)\f{T_0^{3s}}{s}\dd s.
\end{align*}
By moving the line of integration to $\Re(s)=-3$, and using the functional equation \eqref{eqn:FE}, we get
\begin{align*}
  L\left(s+\f{1}{2},\phi_j\right)=\sum_{m\geq 1}\f{A_j(1,m)}{m^{1/2}}V(m,T_0)
  + \sum_{m\geq 1}\f{A_j(m,1)}{m^{1/2}}V_{j}(m,\mu_j),
\end{align*}
where
$$
  V(y,T_0):=\f{1}{2\pi i}\int_{(3)}G(s)\left(\f{y}{T_0^3}\right)^{-s}\f{\dd s}{s},
$$
and
$$
  V_{j}(y,\mu_j):=\int_{(3)}G(s)(yT_0^3)^{-s} \prod_{k=1}^3\f{\Gamma\left(\f{s+1/2+\mu_{j,k}}{2}\right)}{\Gamma\left(\f{-s+1/2-\mu_{j,k}}{2}\right)} \f{\dd s}{s}.
$$
To see the properties of $V$ and $V_{j}$, we use the strategy in \cite{iwaniec2004analytic}.
Obviously, we have
$V(y,T_0)=1+O\left((y/T_0^3)^B\right)$,
and
$V(y,T_0)\ll (y/T_0^3)^{-B}$,
by moving the integration line to $\Re s=-B$ and $\Re s=B$ respectively.
For $V_{j}$, we have $V_{j}(y,\mu_j)\ll (yT)^{-B}$ for any $y\geq 1$,
by using the Stirling formula and moving the integration line
to $\Re s=B/(3\varepsilon)$.
With the help of these, we infer that
\[\label{first moment}
  \begin{split}
    & \sum_{j} \frac{h_{T,M}(\mu_j)}{\cN_j} L(\thf,\phi_j) M_j \\
    & \hskip 50pt = \sum_{j} \frac{h_{T,M}(\mu_j)}{\cN_j} \sum_{\ell} \frac{x_{\ell}}{\ell^{1/2}}
         A_j(1,\ell)
        \sum_{m}\frac{1}{m^{1/2}}
        V\left(m,T_0\right)\overline{A_j(m,1)} + O(T^{-B})\\
    & \hskip 50pt =  \sum_{\ell}
        \frac{x_{\ell}}{\ell^{1/2}}
        \sum_{m}\frac{V(m,T_0)}{m^{1/2}}
        \sum_{j} \frac{h_{T,M}(\mu_j)}{\cN_j}
        A_j(1,\ell) \overline{A_j(m,1)}+O(T^{-B}).
  \end{split}
\]
Then, by the Kuznetsov trace formula, we have
\[
  \begin{split}
   & \sum_{\ell}\frac{x_{\ell}}{\ell^{1/2}}
     \sum_{m}\frac{V(m,T_0)}{m^{1/2}}
     \sum_{j} \frac{h_{T,M}(\mu_j)}{\cN_j} A_j(1,\ell) \overline{A_j(m,1)}\\
   & \hskip 50pt = \sum_{\ell}\frac{x_{\ell}}{\ell^{1/2}}
     \sum_{m}\frac{V(m,T_0)}{m^{1/2}}
     \left(\Delta^{(1)}+\Sigma_4^{(1)}+\Sigma_5^{(1)}+\Sigma_6^{(1)}
     -\mathcal{E}_{\mathrm{max}}^{(1)}-\mathcal{E}_{\mathrm{min}}^{(1)}\right),
  \end{split}
\]
where
\begin{displaymath}
    \begin{split}
      \Delta^{(1)} & := \delta_{\ell, 1} \delta_{1, m}
            \frac{1}{192\pi^5} \int_{\Re \mu = 0} h_{T,M}(\mu) \dd_{\spec}\mu,\\
      \Sigma_{4}^{(1)}& := \sum_{\epsilon  = \pm 1} \sum_{\substack{D_2 \mid D_1\\  m D_1= n_1 D_2^2}}
            \frac{\tilde{S}(-\epsilon, m, 1; D_2, D_1)}{D_1D_2}
            \Phi_{w_4}\!\left(  \frac{\epsilon m}{D_1 D_2} \right),  \\
      \Sigma_{5}^{(1)} & := \sum_{\epsilon  = \pm 1} \sum_{\substack{D_1 \mid D_2\\ D_2 = D_1^2}}
            \frac{\tilde{S}(\epsilon \ell, 1, m; D_1, D_2) }{D_1D_2}
            \Phi_{w_5}\!\left( \frac{\epsilon \ell m}{D_1 D_2}\right),\\
      \Sigma_6^{(1)} & := \sum_{\epsilon_1, \epsilon_2 = \pm 1} \sum_{D_1,  D_2  }
            \frac{S(\epsilon_2, \epsilon_1 \ell, 1, m; D_1, D_2)}{D_1D_2}
            \Phi_{w_6}\!  \left( - \frac{\epsilon_2 D_2}{D_1^2},
            - \frac{\epsilon_1 m \ell D_1}{ D_2^2}\right),
    \end{split}
\end{displaymath}
and
\begin{equation*}
    \begin{split}
       \cE_{\max}^{(1)} & :=  \sum_{g} \frac{1}{2\pi i} \int\limits_{\Re(\mu)=0}
            \frac{h_{T,M}(\mu+\mu_g,\mu-\mu_g,-2\mu)}{\cN_{\mu,g}}
             \overline{B_{\mu,g}(1, m)} B_{\mu,g}(\ell,1) \dd\mu, \\
       \cE_{\min}^{(1)} & := \frac{1}{24(2\pi i)^2} \iint\limits_{\Re(\mu)=0} \frac{h_{T,M}(\mu)}{\cN_{\mu}}
                        \overline{A_{\mu}(1, m)} A_{\mu}(\ell,1) \dd\mu_1 \dd\mu_2.
    \end{split}
\end{equation*}

\subsection{The diagonal term}

By trivial estimation, the contribution of the diagonal term is
\begin{align*}\label{diagonal first moment}
  V(1,T_0)\f{1}{192\pi^5}\int_{\Re \mu = 0} h_{T,M}(\mu) \dd_{\spec}\mu
  =\f{1}{192\pi^5}\int_{\Re \mu = 0} h_{T,M}(\mu) \dd_{\spec}\mu+O(T^{-B})\asymp T^3M^2.
\end{align*}

\subsection{The other terms}

The treatments of the other terms are actually similar and in fact easier than those
in the mollified second moment. So we omit the arguments here.



\begin{thebibliography}{99}
\bibitem{Blomer2008}V. Blomer. {On the central value of symmetric square $L$-functions}. {\em Math. Z.} {260}(4):755--777, 2008.

\bibitem{blomer2015subconvexity}
V.~Blomer and J.~Buttcane.
\newblock On the subconvexity problem for $L$-functions on {${\rm GL}(3)$}.
\newblock {\em arXiv preprint arXiv:1504.02667}, 2015.

\bibitem{buttcane2016spectral}
J.~Buttcane.
\newblock The spectral {K}uznetsov formula on {$SL(3)$}.
\newblock {\em Trans. Amer. Math. Soc.}, 368(9):6683--6714, 2016.

\bibitem{buttcane2016plancherel}
J.~Buttcane and F.~Zhou.
\newblock Plancherel distribution of {S}atake parameters of {M}aass cusp forms
  on {$GL_3$}.
\newblock {\em arXiv preprint arXiv:1611.01253}, 2016.

\bibitem{duke1995}
W.~Duke.
\newblock The critical order of vanishing of automorphic $L$-functions with large level.
\newblock {\em Invent. Math.} 119(1):165--174, 1995.

\bibitem{gelbart-lapid-sarnak}
S.~Gelbart, E.~Lapid, and P.~Sarnak.
\newblock A new method for lower bounds of $L$-functions.
\newblock {\em C. R. Math. Acad. Sci. Paris}. 339(2):91--94, 2004.

\bibitem{goldfeld2006automorphic}
D.~Goldfeld.
\newblock {\em Automorphic Forms and {L}-Functions for the Group {${\rm
  GL}(n,\mathbb{R})$}}, volume~99 of {\em Cambridge Studies in Advanced
  Mathematics}.
\newblock Cambridge University Press, Cambridge, 2006.
\newblock With an appendix by Kevin A. Broughan.

\bibitem{hejhal}
D.~A.~Hejhal.
\newblock The Selberg trace formula for ${\rm PSL}(2,\mathbb{R})$,
\newblock {\em Lecture Notes in Mathematics}, 1001.
\newblock Springer-Verlag, Berlin, 1983.

\bibitem{hoffstein-lockhart}
J.~Hoffstein and P.~Lockhart.
\newblock Coefficients of Maass forms and the Siegel zero.
\newblock {\em Ann. of Math.} (2) 140, no. 1, 161-181, 1994.
With an appendix by D.~Goldfeld, J.~Hoffstein and D.~Lieman.

\bibitem{hoffstein-ramakrishnan}
J.~Hoffstein and D.~Ramakrishnan.
\newblock Siegel zeros and cusp forms.
\newblock {\em Internat. Math. Res. Notices}, no. 6, 279–308, 1995.

\bibitem{iwaniec2004analytic}
H.~Iwaniec and E.~Kowalski.
\newblock {\em Analytic Number Theory}, volume~53.
\newblock American Mathematical Society Providence, 2004.

\bibitem{iwaniec-sarnak1999dirichlet}
H.~Iwaniec and P.~Sarnak.
\newblock Dirichlet $L$-functions at the central point.
\newblock in: {\em Number Theory in Progress}, Vol. 2 (Zakopane-Ko\'{s}cielisko, 1997), 941?952, de Gruyter, Berlin, 1999.

\bibitem{iwaniec-sarnak2000siegel}
H.~Iwaniec and P.~Sarnak.
\newblock The non-vanishing of central values of automorphic $L$-functions and Landau--Siegel zeros.
\newblock {\em Israel J. Math.} 120:155--177, 2000

\bibitem{jacquet-shalika}
H.~Jacquet and J.Shalika.
\newblock A non-vanishing theorem for zeta functions of ${\rm GL}_n$.
\newblock {\em Invent. Math.} 38, no. 1, 1–16, (1976/77).



\bibitem{Khan2010}R. Khan. {Non-vanishing of the symmetric square $L$-function at the central point}. {\em Proc. Lond. Math. Soc.} {100}(3):736--762, 2010.

\bibitem{kim2003functoriality}
H.~Kim.
\newblock Functoriality for the exterior square of {${\rm GL}_4$} and the
  symmetric fourth of {${\rm GL}_2$}.
\newblock {\em J. Amer. Math. Soc.}, 16(1):139--183, 2003.
\newblock With appendix 1 by Dinakar Ramakrishnan and appendix 2 by Kim and
  Peter Sarnak.
\bibitem{KowalskiMichel1999}E. Kowalski, P. Michel. {The analytic rank of $J_0(q)$ and zeros of automorphic $L$-functions}. {\em Duke Math. J.} {100}:503--547, 1999.
\bibitem{KowalskiMichel2000}E. Kowalski, P. Michel. {A lower bound for the rank of $J_0(q)$}. {\em Acta Arith.} {94}(4):303--343, 2000.
\bibitem{KMV2000}E. Kowalski, P. Michel, J. VanderKam. {Mollification of the fourth moment of automorphic $L$-functions and arithmetic applications}. {\em Invent. Math.} {142}:95--151, 2000.
\bibitem{KMV2000b}E. Kowalski, P. Michel and J. VanderKam. {Non-vanishing of high derivatives of automorphic $L$-functions at the center of the critical strip}. {\em J. Reine Angew. Math.} {526}: 1--34, 2000.
\bibitem{KMV2002}E. Kowalski, P. Michel, J. VanderKam. {Rankin-Selberg $L$-functions in the level aspect}. {\em Duke Math. J.} {114}(1):123--191, 2002.


\bibitem{luo2001}
W.~Luo.
\newblock Nonvanishing of $L$-values and the Weyl law.
\newblock{\em Ann. of Math. }(2), 154(2):477--502, 2001.


\bibitem{Luo2015}
W.~Luo.
\newblock {Nonvanishing of the central $L$-values with large weight}.
\newblock {\em Adv. Math.} {285}:220--234, 2015.

\bibitem{Soundararajan2000}
K. Soundararajan.
\newblock {Nonvanishing of quadratic Dirichlet $L$-functions at $s=\hf$}.
\newblock {\em Ann. of Math.} {152}:447--488, 2000.

\bibitem{titchmarsh1986theory}
E.~C. Titchmarsh.
\newblock {\em The theory of the {R}iemann zeta-function}.
\newblock The Clarendon Press, Oxford University Press, New York, second
  edition, 1986.
\newblock Edited and with a preface by D. R. Heath-Brown.

\bibitem{VanderKam1999}J. VanderKam.
\newblock {The rank of quotients of $J_0(N)$}.
\newblock {\em Duke Math. J.} {97}:545--577, 1999.


\bibitem{young2016largesieve}
M.~Young.
\newblock Bilinear forms with {$GL_3$} Kloosterman sums and the spectral large sieve.
\newblock {\em arXiv preprint arXiv:1505.02150v4}, 2016.

\end{thebibliography}
\end{document}